\documentclass[11pt]{amsart}

\usepackage[foot]{amsaddr}

\usepackage{amssymb,amsfonts,amsmath,graphicx,verbatim,eufrak,amsthm,calc}
\usepackage{color}
\usepackage{xcolor}
\usepackage{newfloat}
\usepackage{hyperref}
\usepackage{faktor}

\usepackage{tikz}
\hypersetup{
hidelinks,backref=true,pagebackref=true,hyperindex=true,
    colorlinks=true,       
    linkcolor=violet,          
    citecolor=green,        
    filecolor=magenta,      
    urlcolor=cyan           
}

\colorlet{genial}{black} 
\colorlet{genialsol}{black}

\makeatletter

\newtheoremstyle{genialnumbox}
{7pt}
{7pt}
{\normalfont}
{}
{\small\bf\sffamily\color{genial}}
{\;}
{0.25em}
{%
{\small\sffamily\color{genial}\thmname{#1}}%
{\nobreakspace\thmnumber{\@ifnotempty{#1}{}\@upn{#2}}}
\thmnote{{\nobreakspace\the\thm@notefont\sffamily\bfseries\color{black}\nobreakspace(#3)}} 
}

\newtheoremstyle{blacknumex}
{7pt}
{7pt}
{\normalfont}
{} 
{\small\bf\sffamily}
{\;}
{0.25em}
{%
{\small\sffamily\color{genial}\thmname{#1}}%
{\nobreakspace\thmnumber{\@ifnotempty{#1}{}\@upn{#2}}}
\thmnote{{\nobreakspace\the\thm@notefont\sffamily\bfseries\color{black}\nobreakspace(#3)}} 
}

\newtheoremstyle{blacknumbox} 
{7pt}
{7pt}
{\normalfont}
{}
{\small\bf\sffamily}
{\;}
{0.25em}
{%
{\small\sffamily\color{genial}\thmname{#1}}%
{\nobreakspace\thmnumber{\@ifnotempty{#1}{}\@upn{#2}}}
\thmnote{{\nobreakspace\the\thm@notefont\sffamily\bfseries\color{black}\nobreakspace(#3)}} 
}

\newtheoremstyle{genialnum}
{7pt}
{7pt}
{\normalfont}
{}
{\small\bf\sffamily\color{genial}}
{\;}
{0.25em}
{%
{\small\sffamily\color{genial}\thmname{#1}}%
{\nobreakspace\thmnumber{\@ifnotempty{#1}{}\@upn{#2}}}
\thmnote{{\nobreakspace\the\thm@notefont\sffamily\bfseries\color{black}\nobreakspace(#3)}} 
}

\makeatother

\RequirePackage[framemethod=default]{mdframed} 

\newmdenv[skipabove=7pt,
skipbelow=7pt,
rightline=false,
leftline=false,
topline=false,
bottomline=false,
backgroundcolor=black!5,
linecolor=genial,
innerleftmargin=5pt,
innerrightmargin=5pt,
innertopmargin=10pt,
leftmargin=0cm,
rightmargin=0cm,
innerbottommargin=10pt]{tBox}

\newmdenv[skipabove=7pt,
skipbelow=7pt,
rightline=false,
leftline=false,
topline=false,
bottomline=false,
backgroundcolor=genial!10,
linecolor=genial,
innerleftmargin=5pt,
innerrightmargin=5pt,
innertopmargin=5pt,
innerbottommargin=5pt,
leftmargin=0cm,
rightmargin=0cm,
linewidth=4pt]{eBox}	

\newmdenv[skipabove=7pt,
skipbelow=7pt,
rightline=false,
leftline=true,
topline=false,
bottomline=false,
linecolor=genial!50,
innerleftmargin=5pt,
innerrightmargin=5pt,
innertopmargin=5pt,
leftmargin=0cm,
rightmargin=0cm,
linewidth=4pt,
innerbottommargin=5pt]{dBox}	

\newmdenv[skipabove=7pt,
skipbelow=7pt,
rightline=false,
leftline=false,
topline=false,
bottomline=false,
linecolor=gray,
backgroundcolor=black!5,
innerleftmargin=5pt,
innerrightmargin=5pt,
innertopmargin=5pt,
leftmargin=0cm,
rightmargin=0cm,
linewidth=4pt,
innerbottommargin=5pt]{cBox}

\newmdenv[skipabove=7pt,
skipbelow=7pt,
rightline=false,
leftline=false,
topline=false,
bottomline=false,
linecolor=gray,
backgroundcolor=black!5,
innerleftmargin=5pt,
innerrightmargin=5pt,
innertopmargin=5pt,
leftmargin=0cm,
rightmargin=0cm,
linewidth=4pt,
innerbottommargin=5pt]{pBox}

\newmdenv[skipabove=7pt,
skipbelow=7pt,
rightline=false,
leftline=false,
topline=false,
bottomline=false,
linecolor=genialsol,
innerleftmargin=5pt,
innerrightmargin=5pt,
innertopmargin=0pt,
leftmargin=0cm,
rightmargin=0cm,
linewidth=4pt,
innerbottommargin=0pt]{solBox}	

\theoremstyle{genialnumbox}
\newtheorem{thm1}{Theorem}[section]
\newtheorem{ithm1}[thm1]{$\star$ THEOREM}
\newtheorem{ques1}[thm1]{Question}
\newtheorem{conj1}[thm1]{Conjecture}

\theoremstyle{blacknumex}
\newtheorem{exer}[thm1]{Exercise}
\newtheorem{exer*}[thm1]{$\ast$ Exercise}

\theoremstyle{blacknumbox}
\newtheorem{dfn1}[thm1]{Definition}

\theoremstyle{genialnum}
\newtheorem{cor1}[thm1]{Corollary}
\newtheorem{prop1}[thm1]{Proposition}
\newtheorem{lem1}[thm1]{Lemma}

\newtheorem{exm1}[thm1]{Example}

\newenvironment{theorem}{\paragraph{ } \begin{tBox}\begin{thm1}}{\end{thm1}\end{tBox}}

\newenvironment{exe*}{\paragraph{ } \begin{eBox}\begin{exer*}}{\hfill{\color{genial}
\ensuremath{\diamond\diamond\diamond}}\end{exer*}\end{eBox}}
\newenvironment{definition}{\paragraph{ } \begin{dBox}\begin{dfn1}}{\end{dfn1}\end{dBox}}	
\newenvironment{example}{\paragraph{ } \begin{exm1}}{\hfill{\tiny%
\ensuremath{\bigtriangleup\bigtriangledown\bigtriangleup}}\end{exm1}}
\newenvironment{corollary}{\paragraph{ } \begin{cBox}\begin{cor1}}{\end{cor1}\end{cBox}}	
\newenvironment{ques}{\paragraph{ } \begin{cBox}\begin{ques1}}{\end{ques1}\end{cBox}}	
\newenvironment{conj}{\paragraph{ } \begin{cBox}\begin{conj1}}{\end{conj1}\end{cBox}}	

\newenvironment{proposition}{\paragraph{ } \begin{pBox}\begin{prop1}}{\end{prop1}\end{pBox}}	
\newenvironment{lemma}{\paragraph{ } \begin{pBox}\begin{lem1}}{\end{lem1}\end{pBox}}

\newenvironment{lem*}[1]{\vspace{1ex}\noindent
{\bf Lemma* (#1).} [restatement]  \hspace{0.5em} \em }{ }

\newenvironment{thm*}[1]{\begin{cBox}
\vspace{1ex}\noindent 
{\bf Theorem* (#1).} [restatement]  \hspace{0.5em} }{\end{cBox}}

\theoremstyle{genialnum}

\newtheorem*{clm*}{Claim}

\newenvironment{sol}%
{\begin{solBox}
\par \noindent 
\scriptsize
{\bf Solution to ex:{\color{blue} \arabic{exer}}.}  {\color{red} \ \  :( } \\ }%
{\hfill {\color{blue} :) $\checkmark$} \end{solBox}}

\newcommand{\ENDEXER}{
{\expandafter\comment}
{\expandafter\endcomment}
}

\newtheorem{remark}[thm1]{Remark}

\makeatletter
\renewcommand{\@seccntformat}[1]{\llap{\textcolor{genial}{\csname the#1\endcsname}\hspace{1em}}}                    
\renewcommand{\section}{\@startsection{section}{1}{\z@}
{-4ex \@plus -1ex \@minus -.4ex}
{1ex \@plus.2ex }
{\normalfont\large\sffamily\bfseries}}
\renewcommand{\subsection}{\@startsection {subsection}{2}{\z@}
{-3ex \@plus -0.1ex \@minus -.4ex}
{0.5ex \@plus.2ex }
{\normalfont\sffamily\bfseries}}
\renewcommand{\subsubsection}{\@startsection {subsubsection}{3}{\z@}
{-2ex \@plus -0.1ex \@minus -.2ex}
{.2ex \@plus.2ex }
{\normalfont\small\sffamily\bfseries}}                        
\renewcommand\paragraph{\@startsection{paragraph}{4}{\z@}
{-2ex \@plus-.2ex \@minus .2ex}
{.1ex}
{\normalfont\small\sffamily\bfseries}}

\makeatother

\newcommand{\IP}[1]{\left\langle #1 \right\rangle}

\newcommand{\Integer}{\mathbb{Z}}

\newcommand{\Z}{\Integer}
\newcommand{\N}{\mathbb{N}}
\newcommand{\Q}{\mathbb{Q}}
\newcommand{\R}{\mathbb{R}}

\newcommand{\ie}{{\em i.e.\ }}
\newcommand{\eg}{{\em e.g.\ }}

\def\squareforqed{\hbox{\rlap{$\sqcap$}$\sqcup$}}
\def\qed{\ifmmode\squareforqed\else{\unskip\nobreak\hfil
\penalty50\hskip1em\null\nobreak\hfil\squareforqed
\parfillskip=0pt\finalhyphendemerits=0\endgraf}\fi}

\newcommand{\ignore}[1]{ }

\newcommand{\p}{\partial}

\newcommand{\vphi}{\varphi}

\newcommand{\F}{\mathbb{F}}

\newcommand{\AND}{\qquad \textrm{and} \qquad}

\newcommand{\mbf}[1]{\mathbf{#1}}

\newcommand{\define}[1]{\textbf{#1}}

\newcommand{\Gr}{\mathsf{Gr}}

\title{Detecting virtual homomorphisms via Banach metrics}
\author{Liran Ron-George}
\author{Ariel Yadin}
\address{Department of Mathematics, Ben-Gurion University of the Negev}
\email{lirar@post.bgu.ac.il, yadina@bgu.ac.il}
\thanks{We thank C.\ Bodart, Y.\ Glasner and A.\ Karlsson for useful discussions and insights. 
We also thank an anonymous referee for spotting a mistake
in a previous version of Lemma \ref{lem:Cayley graphs for VA}.
Research supported by 
the Israel Science Foundation, grant no.\ 954/21.
The first author 
is also partially supported by the Israel Science Foundation, grant no.\ 1175/18}

\begin{document}

\begin{abstract}
We introduce the notion of {\em Banach metrics} on finitely generated infinite groups. This extends the notion
of a Cayley graph (as a metric space).
Our motivation comes from trying to detect the existence of virtual homomorphisms into $\Z$.
We show that detection of such homomorphisms through metric functional boundaries of Cayley graphs isn't always possible.
However, we prove that 
it is always possible to do so through a metric functional boundary of some Banach metric on the group.
\end{abstract}

\maketitle

\section{Introduction}

Gromov's theorem regarding groups of polynomial growth \cite{Gromov81} is widely considered 
a cornerstone of geometric group theory.  In that paper Gromov proved that any finitely generated group
of polynomial growth is virtually nilpotent.  The proof has two main stages to it.  The first, is to show that in any 
finitely generated group of polynomial growth there is a finite index subgroup with a surjective homomorphism onto 
the additive group of integers. (This is usually done by constructing a representation of the group.) 
We call such a homomorphism a {\em virtual homomorphism} for short, see below, Definition \ref{dfn:vir hom}.
(``Virtual'' because the homomorphism is only defined on a finite index subgroup.)
The second stage in Gromov's proof is an induction on the degree of polynomial growth showing that such groups must be virtually nilpotent.

This was not the first time growth of finitely generated groups was studied.  Already in 1968 Milnor and Wolf considered 
the polynomial growth setting for solvable groups in \cite{Milnor68, Wolf68}, proving that any finitely generated
solvable group is either of exponential growth or virtually nilpotent.  (Finitely generated nilpotent groups 
were shown to have polynomial growth by Wolf in \cite{Wolf68}, 
and later Bass and Guivarc'h \cite{Bass72, Guivarch73} computed the exact 
degree of polynomial growth of a finitely generated nilpotent group.)
In the same year 1968, Milnor \cite[Problem 5603]{advancedproblems} asked if there exist finitely generated groups
that are not of exponential growth and not of polynomial growth (the so called groups of {\em intermediate growth}).
This was finally answered affirmatively by Grigorchuk \cite{Grigorchuk80, Grigorchuk84}.
Grigorchuk also conjectured \cite{Grigorchuk90} that there is a ``gap'' in the 
possible growth functions of finitely generated groups; 
namely, all finitely generated groups of small enough growth must actually be of polynomial growth, 
see Conjecture \ref{conj:gap} for a precise statement.

One naive thought on how to attack Grigorchuk's gap conjecture would be to reproduce Gromov's strategy:
prove that any finitely generated group of small growth admits a finite index subgroup with a surjective homomorphism onto $\Z$, and then use some 
sort of induction argument.  
Detecting surjective homomorphisms onto $\Z$ provides a lot of information about a group.
If $G/N \cong \Z$ for some normal subgroup $N \lhd G$, then it is not difficult to see that $G \cong \Z \ltimes N$,
where $\Z$ acts on $N$ by some automorphism of $N$. (This may be compared to the Gromoll splitting theorem \cite{Gromoll},
in a different context.)

We are thus motivated to try and understand how to find virtual homomorphisms on groups.
One suggestion by A.\ Karlsson \cite{karlsson2008ergodic} was to consider the action of the group $G$ on a boundary of the {\em 
metric-functional compactification} of $G$.  Elements of this compactification are functions from $G$ into $\Z$, 
and a finite orbit for the canonical action provides a virtual homomorphism.  This will be explained precisely in 
Section \ref{scn:action on boundary}.

Naively one may wish to consider compactifications of Cayley graphs of the group, since these metric spaces are 
the ones used to measure growth, and intimately connect the geometric and algebraic properties of the group.
However, we show in Corollary \ref{cor:free group no detection} 
that on the free group, although there exist surjective homomorphisms onto $\Z$,
the metric-functional  boundaries of Cayley graphs of the free group never contain a finite orbit.

We therefore extend the notion of Cayley graphs to a broader class of metric spaces on a group, which we dub {\em Banach metrics}
(see Definition \ref{dfn:Banach metric}).  These are quasi-isometric to Cayley graphs, so still capture the correct geometry, 
but are general enough metric spaces so that the metric-functionals can still detect virtual homomorphisms.
Our main result, Theorem \ref{thm:detection}, states that a finitely generated infinite group $G$ admits a virtual homomorphism 
if and only if there exists some Banach metric on $G$ such that its metric-functional boundary contains a finite orbit (and therefore
functions in this orbit are virtual homomorphisms).

Another issue with working only with Cayley graphs is that 
this notion is not rich enough for some basic operations 
in geometric group theory.   Specifically, 
the restriction of the metric to subgroups, even to those of finite index, 
does not typically result in a Cayley graph.  
In contrast, we show in Theorem \ref{thm:finite index BM} that the restriction of a Banach metric to a finite index subgroup
always results in a Banach metric on that subgroup.

Banach metrics are much more flexible than Cayley graphs, as can be seen by the construction in Lemma \ref{lem:BM construction}.
The above mentioned results, Theorems \ref{thm:detection} and \ref{thm:finite index BM}, motivate further exploration of such metrics and their 
connection to growth in small groups.
To summarize the comparison between Banach metrics and Cayley graphs:
\begin{itemize}
\item Banach metrics pass to finite index subgroups whereas Cayley graphs do not.
\item Cayley graphs cannot always ``detect'' virtual homomorphisms, however Banach metrics always do.
\item Banach metrics are much more flexible, which may be an advantage for creatively constructing useful ones.
\end{itemize}

We now move to precisely define the above notions and state our results.

\section{Background}

\subsection{Metric-functionals}

Let $(X,d)$ be a metric space with a base point $x_0 \in X$ and denote $L(X,d)$ the set of all functions $h:X \to \mathbb{R}$ such that $h$ is $1$-Lipschitz (i.e. $|h(x)-h(y)| \leq d(x,y)$ for all $x,y \in X$) and $h(x_0)=0$. Equip $L(X,d)$ with the topology of pointwise convergence 
and note that $L(X,d)$ is compact by Tychonoff's theorem. 
The set $X$ embeds into $L(X,d)$ by identifying $x \in X$ with the so called {\em Busemann function} $b_x:X \to \mathbb{R}$ given by $b_x(y)=d(x,y)-d(x,x_0)$. We denote the closure of $\{b_x \ | \  x \in X\}$ in $L(X,d)$ by $\overline{(X,d)}$ and define the \define{metric-functional boundary} of $(X,d)$ to be
$$\partial (X,d)=\overline{(X,d)} \setminus \{b_x \ | \  x \in X\}.$$
The elements of $\p (X,d)$ are called {\em metric-functionals}, and they play a role corresponding to linear functionals, but on 
general metric spaces that do not afford a linear structure.  See \eg \cite{karlsson2021linear, karlsson2024metric} and references therein.

If we replace the above topology on $L(X,d)$ with uniform convergence on compacts, we would arrive at an analogous compact space,
which has been considered extensively (see \eg 
\cite{arosio2024horofunction, busemann2005geometry,
gromov1981hyperbolic,
karlsson2001non,
karlsson2021linear,  karlsson2021hahn, 
karlsson2024metric}, and many other texts).  
In this case elements in the boundary are sometimes called {\em horofunctions}.
Since we will always consider cases where $X$ is countable and discrete, there is no distinction in this paper 
between horofunctions and metric functionals.

\subsection{Cayley graphs}

Let $S$ be a finite symmetric generating set for a group $G$.
That is, $G = \IP{S}$, $|S| < \infty$, $S=S^{-1}$.
We consider the {\em Cayley graph} of $G$ with respect to $S$.
This is a graph whose vertices are the elements of $G$, and edges are given by the relation $x \sim y \iff x^{-1} y \in S$.
Since $S$ is symmetric this defines a graph, denoted $\Gamma(G,S)$, 
and therefore a metric space, where the metric is the graph distance
(which is incidentally the word metric with respect to $S$).
If $d_S$ denotes the graph metric, then it is easy to see that $d_S$ is {\em left-invariant}, \ie $d_S(zx,zy) = d_S(x,y)$ for all $x,y,z \in G$.
It is convenient to use the notation $|x|_S = d_S(x,1)$.

Recall that two metric spaces $(X,d), (Y,\rho)$ are \define{quasi-isometric}
if there exists a \define{quasi-isometry} $\vphi : X \to Y$. This means that
there exists $C>0$ such that 
for any $y \in Y$ there exists $x \in X$ 
such that $\rho(\vphi(x), y) \leq C$ and 
also for any $x,x' \in X$ 
$$ C^{-1} d(x,x') - C \leq \rho ( \vphi(x), \vphi(x') ) \leq C d(x,x') + C $$
It is a fact that quasi-isometries provide an equivalence relation between metric spaces, see \eg \cite[Chapter 3]{Gabor}.

A simple exercise shows that for two finite symmetric generating sets $S,T$ of a group, 
the corresponding metrics $d_S, d_T$ are {\em bi-Lipschitz}.  That is, there exists some constant $C = C(S,T)>0$ such that
$$ C^{-1} d_S(x,y) \leq d_T(x,y) \leq C d_S(x,y) $$ 
for all $x,y \in G$.
Specifically, all Cayley graph metrics on the same group are quasi-isometric to one another.

Moreover, if $H \leq G$ is a subgroup of finite index $[G:H] < \infty$ of a finitely generated group $G$, then $H$ is finitely 
generated as well (see \eg \cite[Exercise 1.61]{HFbook}).  The Cayley graphs of $G$ and of $H$ are quasi-isometric as well
(see \eg \cite{Gabor}).

Cayley graphs are geodesic metric spaces. 
One can show that this property implies that metric-functionals of infinite Cayley graphs are always unbounded from below.  In fact, for any $h \in \p (G,d_S)$ and any integer $r \geq 0$ there exists $x \in G$
with $h(x) = -|x|_S = -r$.  (For a proof see Lemma \ref{lem:integer valued geodesic metrics} below.) 
This property separates metric-functionals from interior points $b_x$, $x \in G$, because one readily verifies that $b_x(y) \geq -|x|_S$
for all $x,y\in G$.

Another property of Cayley graphs is that they are {\em proper} metric spaces; \ie balls are compact (finite in our case).
In fact, any geodesic, integer valued, proper, left-invariant metric on a group $G$ can be easily shown to be a metric arising from a Cayley graph.
We call such metrics {\em Cayley metrics} (these are also known as {\em word metrics}).

These properties imply that for a converging sequence $b_{x_n} \to h \in \overline{ (G,d_S)}$,
we have a dichotomy:  Either $x_n=x$ for all large enough $n$, and $h=b_x$, or $h \in \p (G,d_S)$ and $|x_n| \to \infty$,
but both cannot hold simultaneously. 
So boundary points in $\p (G,d_S)$ are indeed ``points at infinity'', and the structure $\overline{ (G,d_S)}$ is a 
compactification of $G$.

For an example see Example \ref{exm:Zd boundary} below.

\subsection{Action on the boundary}

\label{scn:action on boundary}

Let $d$ be any metric on a group $G$ and let $x_0=1$.
$G$ acts naturally on $L(G,d)$ by $x.h(y) = h(x^{-1} y) - h(x^{-1})$.
One readily verifies that this is a continuous action, and that $\p (G,d)$ is $G$-invariant.
(Note that $x.b_y=b_{xy}$.)

Assume that $h \in L(G,d)$ is a fixed point for the $G$-action.
That is, $x.h=h$ for all $x \in G$.  This precisely means that $h$ is a homomorphism from $G$ into $\R$.

Furthermore, if $h \in L(G,d)$ has a finite orbit $|G.h| < \infty$, then by taking $H$ to be the stabilizer of $h$,
we obtain a finite index $[G:H] < \infty$ subgroup, such that the restriction of $h$ to $H$ is a homomorphism from $H$ into $\R$.
If $h \big|_H \equiv 0$ then $h$ is a bounded function on $G$.  Thus, if the situation is such that $h \in \p (G,d)$
and all metric-functionals are unbounded, then we have obtained a non-trivial homomorphism from the finite index subgroup $H$
into $\R$.  So $h(H)$ is an infinite finitely generated abelian group, implying that $H$ admits some surjective homomorphism onto $\Z$.

\begin{definition} \label{dfn:vir hom}
Let $G$ be a group.  A \define{virtual homomorphism} on $G$ is a function $\vphi :G \to \Z$
such that there exists a finite index subgroup $[G:H] < \infty$ for which the restriction $\vphi \big|_H$ is a non-trivial
homomorphism. 
\end{definition}

Let us remark that in many texts a group admitting a virtual homomorphism is called a \define{virtually indicable} group.

Thus, we have seen that if $d$ is an integer valued metric on $G$ and $h \in \p(G,d)$ is unbounded and has a finite orbit $|G.h| <\infty$, then 
$h:G \to \Z$ is such a virtual homomorphism.
However, when $d=d_S$ is the metric of some Cayley graph,
it is not necessarily true that any virtual homomorphism can be found this way, 
as the next theorem shows.

\begin{theorem} \label{thm:hyperbolic no detection}
Let $G$ be a finitely generated Gromov hyperbolic group, which is not virtually cyclic.
Let $d_S$ be the metric of some Cayley graph on $G$.

Then, there are no finite orbits in $\p (G,d_S)$.
\end{theorem}

A finitely generated free group is Gromov hyperbolic, which implies the following corollary.

\begin{corollary} \label{cor:free group no detection}
Let $\F_d$ be the free group of $d \geq 2$ generators. Let $d_S$ be the metric of some Cayley graph on $\F_d$.
Then, there are no finite orbits in $\p (\F_d,d_S)$.
\end{corollary}

For the definition of Gromov hyperbolic groups and for a proof of Theorem \ref{thm:hyperbolic no detection},
see Section \ref{scn:no detection}.

\subsection{Banach metrics}

In light of Theorem \ref{thm:hyperbolic no detection}, if we consider the 
``detection problem'' for virtual homomorphisms via the metric-functionals, we must broaden the types of possible metrics we use
to more than just Cayley graphs.
As mentioned above, metric-functionals play an analogous role to linear functionals.
In the linear setting (vector spaces), the only bounded function which is a linear functional is the $0$ functional (trivial functional).
In the general metric space setting, the properties of triviality and boundedness become distinct. 
To preserve our analogy to the linear world we require that metric functionals are unbounded.
This is in addition to other metric properties, as detailed in the following definition.

\begin{definition} \label{dfn:Banach metric}
Let $G$ be a finitely generated infinite group.
A metric $d$ on $G$ is called a \define{Banach metric} if it has the following properties:
For all $x,y,z \in G$,
\begin{enumerate}
\item $d(x,y) \in \N$ (integer valued).
\item $d(zx,zy)= d(x,y)$ (left-invariant).
\item For any $r>0$ the ball $B_d(r) = \{ x \in G \ : \ d(x,1) \leq r \}$ is finite (\ie $(G,d)$ is a proper metric space).
\item $(G,d)$ is quasi-isometric to a Cayley graph (group geometry).
\item Any metric-functional $h \in \p (G,d)$ is an unbounded function.
\end{enumerate}
\end{definition}


It follows from Lemmas \ref{lem:integer valued geodesic metrics} and \ref{lem:boundary points are at infinity}
below that any Cayley metric is a Banach metric.
However, Banach metrics are more general than Cayley graphs.

\begin{example}
A preliminary example: 
Let $G$ be a group with a finite symmetric generating set $S$.
It is not difficult to verify that $M \cdot d_S$ is a Banach metric for any positive integer $M$.
Also, if $M>1$ then $M \cdot d_S$ cannot be a Cayley metric for the technical reason that Cayley metrics 
have minimal distance $1$, and this metric has minimal distance $M$.
\end{example}

A more intriguing example will be given in Lemma \ref{lem:BM construction}.

%
%
%

\subsection{Finite index subgroups}

When discussing metrics on groups, it makes sense to consider 
subgroups as subspaces of the original metric space (by inducing the metric on 
the subgroup). 
However, in the Cayley graph setting, 
this does not result in the same structure.  Typically, the induced metric on a subgroup will not be a Cayley metric.  
For example, it may fail to be a geodesic metric.  This is the case even for subgroups of finite index.
Thus, one sees that the category of Cayley graphs may not be useful if we wish to permit ourselves to pass freely to finite index subgroups
(as is usually the situation when considering geometric properties of the group).

Contrary to the situation for Cayley graphs, we prove that Banach metrics do induce Banach metrics on finite index subgroups.
This stability is another motivation to consider these more general types of metrics. 

\begin{theorem} \label{thm:finite index BM}
Let $G$ be a finitely generated group, and let $H \leq G$ be a finite index subgroup $[G:H] < \infty$.
Let $d$ be a Banach metric on $G$, and let $d_H$ be the induced metric on $H$ (as a subset).

Then, $d_H$ is a Banach metric on $H$.
\end{theorem}

The proof of Theorem \ref{thm:finite index BM} is in Section \ref{scn:finite index}.

\subsection{Detection of virtual homomorphisms}

The notion of a Banach metric is not just a generalization, but, as mentioned, it is useful for detecting virtual homomorphisms.
This is our main result.

\begin{theorem} \label{thm:detection}
Let $G$ be a finitely generated group.  
The following are equivalent:
\begin{itemize}
\item $G$ admits some virtual homomorphism.
\item There exists a Banach metric $d$ on $G$ and some $h \in \p (G,d)$ such that $h$ has a finite orbit $|G.h| < \infty$.
\end{itemize}

Moreover, if $G$ admits an actual homomorphism onto $\Z$, then $h \in \p (G,d)$ above can be chosen 
so that it is a fixed point of $G$.
\end{theorem}

The proof of Theorem \ref{thm:detection} is at the end of Section \ref{scn:detection}.

\subsection{Nilpotent groups}

It is well known that any finitely generated nilpotent group admits a homomorphism onto $\Z$ (indeed it always has an infinite 
Abelianization). Thus, virtually nilpotent groups always admit some virtual homomorphism.
Walsh \cite{Walsh} has shown that in {\em any} Cayley graph of a nilpotent group there is a 
finite orbit in the metric-functional boundary (see also \cite{develin}).  To our knowledge, there is no such result for 
{\em virtually} nilpotent groups.

As an immediate consequence of Theorem \ref{thm:detection} we have:

\begin{corollary} 
\label{cor:vir nilpotent}
Let $G$ be a finitely generated virtually nilpotent group. There exists some Banach metric $d$ on $G$ 
with a finite orbit in $\p (G,d)$.
\end{corollary}

Corollary \ref{cor:vir nilpotent} is new even for virtually Abelian groups, as far as we know.
In order to prove Theorem \ref{thm:detection} we are required to specifically analyze virtually Abelian groups.
As part of the proof of Theorem \ref{thm:detection}, 
we prove a stronger version of Corollary \ref{cor:vir nilpotent} for the (special) case of virtually Ablelian groups, as follows.

\begin{theorem}
\label{thm:vir Abelian}
Let $G$ be a finitely generated virtually Abelian group. There exists some Cayley metric $d_S$ on $G$ 
with a finite orbit in $\p (G,d_S)$.
\end{theorem}

Theorem \ref{thm:vir Abelian} is proven in Section \ref{scn:detection},
see Corollary \ref{cor:vir Abelian fin orbit} there.

While this work was being prepared, Bodart \& Tashiro uploaded 
a preprint to the arXiv where they conjectured 
that a group is virtually Abelian if and only if it has some Cayley metric with a countable metric-functional boundary, see \cite[Conjecture 1(a)]{BT24}.
Through personal communication Bodart \& Tashiro have informed us that they can most likely prove the ``only if'' part: any Cayley graph of a virtually Abelian group has a countable metric-functional boundary.
As already observed by Karlsson \cite{karlsson2008ergodic}, 
by considering a stationary measure on the boundary, one obtains that a countable boundary implies the existence of a finite orbit (take a maximal atom of this measure).

It is not difficult to see that the proof of Theorem 
\ref{thm:vir Abelian} and Corollary \ref{cor:vir Abelian fin orbit} actually proves that for a virtually nilpotent group there always exists 
{\em some} Cayley graph with a countable metric-functional boundary.
We have chosen not to expand on this too much, as Bodart \& Tashiro's yet unpublished result is stronger than our Theorem \ref{thm:vir Abelian},
since it asserts the same for {\em any} Cayley graph of a virtually Abelian group.

\subsection{Gap conjecture}

We recall the usual partial order on monotone functions.
For monotone non-decreasing $f,h : \N \to [0,\infty)$ we write $f \preceq h$ if there exists $C>0$ such that
$f(n) \leq C h(Cn)$ for all $n \in \N$.
This provides an equivalence relation on such functions by $f \sim h$ if and only if $f \preceq h$ and $h \preceq f$.

The \define{growth} of a finitely generated group is the equivalence class of the function $\mathsf{gr}(r) = |B(1,r)|$,
where $B(1,r)$ is the ball of radius $r$ in some fixed Cayley graph.
Since different Cayley graph metrics are quasi-isometric, they provide the same equivalence class of growth,
so this definition does not depend on the specific choice of Cayley graph.

As mentioned in the introduction, the following has been conjectured by Grigorchuk \cite{Grigorchuk90}.

\begin{conj}[Grigorchuk's gap conjecture]
\label{conj:gap}
Let $G$ be a finitely generated group of growth $\preceq r \mapsto \exp ( r^\alpha )$ for some $\alpha < \tfrac12$.

Then, $G$ is virtually nilpotent (and so actually has polynomial growth).
\end{conj}

In light of our main result, Theorem \ref{thm:detection}, and the fact that virtually nilpotent groups always admit virtual homomorphisms,
we conjecture the following (logically weaker) conjecture.

\begin{conj}[Weak gap conjecture]
Let $G$ be a finitely generated group of growth $\preceq r \mapsto \exp ( r^\alpha )$ for some $\alpha < \tfrac12$.

Then, there exists a Banach metric $d$ on $G$ with a finite orbit in $\p (G,d)$.
\end{conj}

\subsection{Open questions}

Let us conclude this section with some open questions for possible further research.

\begin{ques} \label{ques:type is invariant}
Let $S,T$ be two finite symmetric generating sets for a group $G$.
Show that $\p (G,d_S)$ is countable if and only if $\p (G,d_T)$ is countable.
\end{ques}

\begin{remark}
One may wonder if Question \ref{ques:type is invariant} can be true for a wider class of metrics, \eg for Banach metrics, and not just Cayley metrics.

The following example shows this does not hold.

Consider $G = \Z^d$ and $D(x,y) : = \lceil \| x- y\|_2 \rceil$.
This is easily seen to be a left-invariant, integer valued, proper metric,
which is quasi-isometric to a Cayley graph (the standard Cayley graph is just the $L^1$-metric).  

Assume that $x_n \in \Z^d$ are such that $\|x_n\|_2 \to \infty$ 
and $\frac{1}{ \|x_n\|_2 } x_n \to v \in \R^d$
(recall that the $L^2$ unit ball in $\R^d$ is compact).
It is a simple calculation to verify that for any $y \in \Z^d$ we have 
$$ \lim_{n \to \infty} \big( \| x_n - y \|_2 - \| x_n \|_2 \big) = - \IP{ v, y } . $$

Metric-functionals in $\p (\Z^d,D)$ are unbounded because if
$b_{x_n} \to f \in \p (\Z^d,D)$, then by passing to a subsequence 
we can assume without loss of generality that $\frac{1}{\|x_n\|_2} x_n \to v$
for some unit vector $\|v\|_2=1$.  We then have that for any $y \in \Z^d$,
$$ | f(y) + \IP{v,y} | = 
\lim_{n \to \infty} \big| D(x_n,y) - \|x_n-y\|_2 + \|x_n\|_2 - D(x_n,0)  \big|
\leq 2 . $$
So $f$ is unbounded.

Now, assume that $(x_n)_n , (z_n)_n$ are sequences in $\Z^d$ such that 
$\frac{1}{\|x_n\|_2} x_n \to v$ and $\frac{1}{\|z_n\|_2} x_n \to w$
for some unit vectors $\|v\|_2=\|w\|_2=1$.
By passing to a subsequence, we may assume without loss of generality that
$b_{x_n} \to f \in \p (\Z^d,D)$ and $b_{z_n} \to h \in \p (\Z^d , D)$.
(Here the Busemann functions are with respect to the metric $D$.)

If $v \neq w$, then we may choose some $u \in \Q^d$ such that 
$\IP{ v,u } < 0$ and $\IP{w,u} > 0$.  Since $u \in \Q^d$, there exists 
large enough $K \in \N$ such that $y = K u \in \Z^d$ and such that
$\IP{ v,y } < -1$ and $\IP{ w,y} > 1$.
We then have that
\begin{align*}
f(y) & \geq \liminf_{n \to \infty} \big( \| x_n-y \|_2 - \|x_n\|_2 -1 \big) 
= - \IP{ v,y } - 1 > 0 , \\
h(y) & \leq \limsup_{n \to \infty} \big( \|z_n-y\|_2 +1 - \|z_n\|_2 \big) 
= - \IP{ w,y} + 1 < 0 .
\end{align*}
So $h(y) \neq f(y)$.

This shows that there is an injective mapping from the unit sphere in $\R^d$
into the metric-functional boundary $\p (\Z^d,D)$, 
so this boundary must be uncountable.

On the other hand,
in \cite{develin} it is shown that any Cayley metric on $\Z^d$ has a countable metric-functional boundary, composed entirely of Busemann points.
\end{remark}

\begin{ques}
Assume that $H \lhd G$ has finite index $[G:H]<\infty$.
Let $d_G, d_H$ be Cayley metrics on $G,H$ respectively.
Show that if $\p (H,d_H)$ is countable, then $\p (G,d_G)$ is also countable.
\end{ques}

\begin{ques}
(Conjecture 1(a) in \cite{BT24})
Show that the metric-functional boundary of a non-virtually Abelian group is uncountable.

Perhaps it is easier to start with ``bigger'' groups:
Show that Cayley graphs of non-amenable / exponential growth / non-virtually nilpotent groups have uncountable metric-functional boundaries.
\end{ques}

Recall that Corollary \ref{cor:vir nilpotent} states that for any virtually nilpotent group there exists some Banach metric with a finite orbit in the boundary.
Theorem \ref{thm:vir Abelian} states that in a virtually Abelian group one can do this with a Cayley graph.  Somewhat frustratingly we do currently know how to prove the following:

\begin{conj}
Let $G$ be a finitely generated virtually nilpotent group.  Let $d_S$ be any Cayley metric on $G$.  Then the metric-functional boundary $\p (G,d_S)$ contains a finite orbit.
\end{conj}

\begin{ques}
Does there exist a non-virtually cyclic group $G$ with a Banach metric $d$
and a finite metric-functional boundary $|\p (G,d) | < \infty$?
\end{ques}

\begin{remark}
In ongoing work of the authors with C.\ Bodart, it seems we can prove that 
if there exists a Cayley graph with a finite metric-functional boundary then the group is virtually cyclic.  (The other direction is known, see \cite{RY23}.)
It is not clear if the weaker assumption of a Banach metric with a finite boundary suffices.
\end{remark}

Finally, note that in the definition of a Banach metric we require 
that it is quasi-isometric to a Cayley metric, and that all metric-functionals are unbounded.  It is not immediately obvious if the second condition is not superfluous.

\begin{ques}
Let $G$ be a finitely generated group, and let $d$ be an integer valued, left-invariant, proper metric on $G$ which is quasi-isometric to some (any) Cayley metric.  Is $d$ a Banach metric? \ie are all metric-functionals in $\p (G,d)$ unbounded?
\end{ques}

\section{Metric functionals and quotient groups}

\subsection{Basic properties of metric functionals}

We include some basic properties of metric functionals which we will require.
The proofs are well known, and we include them only for completeness.

Recall that a metric is \define{geodesic} if there is a geodesic path connecting any two points.
Specifically, for an integer valued metric $d$ on $G$, we say that $d$ is \define {geodesic} if for any $x,y \in G$
there exist $x=z_0 , z_1 , \ldots, z_n=y$ such that $d(z_j , z_k)= |j-k|$ for all $0 \leq j,k \leq n$.

%
%
%
%
%
%

\begin{lemma} \label{lem:integer valued geodesic metrics}
Let $d$ be an integer valued proper geodesic metric on $X$.
Fix a base point $x_0 \in X$.
Let $(x_n)_n$ be a sequence such that $d(x_n,x_0) \to \infty$.

Then, for any $r \in \N$ there exists $x \in X$ and an infinite subset $I  \subset \N$ 
such that $b_{x_n}(x) = -d(x,x_0) = -r$ for all $n \in I$.

As a consequence, if $b_{x_n} \to f \in \overline{ (X,d) }$ for a sequence such that $d(x_n,x_0) \to \infty$,
then for every $r \in \N$ there exists $x \in X$ with $d(x,x_0)=r = -f(x)$.
Specifically the function $f$ is unbounded from below.   
\end{lemma}

\begin{proof}
Since the metric is integer valued, the topology induced by the metric is discrete.
As $X$ is proper, balls must be finite sets.  

Fix $r \in \N$ and let $S = \{ x \ : \ d(x,x_0) = r \}$.

Let $y \in X$ be such that $d(y,x_0) > r$.
We assumed that $d$ is geodesic, so we may choose a finite geodesic from the base point $x_0$ to $y$;
\ie a finite sequence $x_0 = z_0 , z_1, \ldots, z_m = y$ such that $d(z_j , z_k)= |j-k|$ for all $0 \leq j,k \leq m$.
Consider the point $w = z_r$. Since $d(w,x_0) = d(z_r,z_0) = r$,
we have that $w \in S$ and 
$$ b_y(w) = d(z_r,z_m) - d(z_0,z_m) = - r = - d(w,x_0) . $$

Thus, for any $n$ such that $d(x_n,x_0)>r$, there exists some $w_n \in S$ such that $b_{x_n}(w_n) = - r = - d(w_n,x_0)$.
Since $S$ is finite, there must exist some $x \in S$ such that $w_n = x$ for infinitely many $n$.
Setting $I = \{ n \ : \ w_n = x\}$ completes the proof of the first assertion.

For the second assertion assume that $b_{x_n} \to f$ and $d(x_n,x_0) \to \infty$.
Fix $r \in \N$. The first assertion tells us that for some infinite subset $I \subset N$ 
we have $b_{x_n}(x) = -r$ for all $n \in I$ and some $x$.
This implies that 
$$ f(x) = \lim_{n \to \infty} b_{x_n}(x) = \lim_{I \ni n \to \infty} b_{x_n} (x) = - r . $$
\end{proof}

\begin{lemma}
\label{lem:boundary points are at infinity}
Let $d$ be an integer valued proper metric on $X$.
Fix a base point $x_0 \in X$.

Then, if $b_{x_n} \to h \in \p (X,d)$, it must be that $d(x_n, x_0) \to \infty$.
\end{lemma}

\begin{proof}
Assume that $(d(x_n,x_0))_n$ is a bounded sequence and that $b_{x_n} \to h$.  We will show that $h \not\in \p(X,d)$,
that is $h = b_x$ for some $x \in X$.

As before, the topology induced by the metric is discrete, and being proper, balls must be finite sets.
So there is some finite set $B$ such that $x_n \in B$ for all $n$.
Hence there must be $x \in B$ such that $x_n=x$ for infinitely many $n$. 
As $b_{x_n} \to h$, it must be that $h=b_x$.
\end{proof}

\subsection{Banach metric construction}

We now move to construct a Banach metric by combining two Cayley graphs in the right way.
This construction exhibits how Banach metrics offer more flexibility than just Cayley graphs.
It will be central to the proof of Theorem \ref{thm:detection}.

\begin{lemma} \label{lem:BM construction}
Let $G$ be finitely generated infinite group, and let $\pi: G \to H$ be a surjective homomorphism.  

Suppose that $d_G$ is a Cayley metric on $G$ and $d_H$ is a Cayley metric on $H$.

Write $| x|_G = d_G(x,1_G)$ and $|q|_H = d_H(q,1_H)$ (here $1_G,1_H$ denote identity elements in the respective groups).
Assume that there exists $C \geq 1$ such that $|\pi(x)|_H \leq C|x|_G$ for every $x \in G$, 
and also for any $q \in H$ there exists $x \in \pi^{-1}(\{q\})$ such that $|x|_G \leq C |q|_H$.

Fix an integer $M > C$ and define 
$$ D(x,y) = \max \{ d_G(x,y) , M \cdot d_H(\pi(x), \pi(y)) \} $$

Then, $D$ is a Banach metric on $G$.
\end{lemma}

\begin{proof}
It is easy to verify that $D$ is indeed a metric, which is proper, integer valued, and  left-invariant.

Note that $d_G(x,y) \leq D(x,y) \leq CM d_G(x,y)$, implying that the identity map on $G$ is a quasi-isometry between $(G,D)$ and $(G,d_G)$.
Since $d_G$ is a Cayley metric, $(G,D)$ is quasi-isometric to a Cayley graph.

This verifies the first $4$ properties of a Banach metric from Definition \ref{dfn:Banach metric}.

Denote $1=1_G$ and $|x|_D = D(x,1)$.

To differentiate between the different metrics, we denote $b_x^D(y) = D(x,y) - |x|_D$,
$b_x^G(y) = d_G(x,y) - |x|_G$ and $b_q^H(p) = d_H(q,p) - |q|_H$.

We now prove the fifth property in Definition \ref{dfn:Banach metric}. 
Let $F \in \p (G,D)$.
Choose some sequence $(g_n)_n$ such that $b_{g_n}^D \to F$.
We know by Lemma \ref{lem:boundary points are at infinity} that $|g_n|_D \to \infty$.

We have $3$ cases:

{\bf Case I.}
$$ \limsup_{n \to \infty} \big( |g_n|_G - M \cdot |\pi(g_n)|_H \big)  = \infty $$
In this case, without loss of generality (by passing to a subsequence), we can assume that $|g_n|_G - M \cdot |\pi(g_n)|_H \to \infty$.
Since $|g_n|_G \to \infty$,  and since $d_G$ is assumed to be geodesic, we can use  Lemma \ref{lem:integer valued geodesic metrics},
so that by passing to a further subsequence, we may assume without loss of generality that 
$b_{g_n}^G \to f \in \overline{ (G,d_G) }$, and $f$ is unbounded from below.

Fix some $r \in \N$.  Choose $x$ so that $f(x) \leq -r$.
Since $|g_n|_G - M \cdot |\pi(g_n)|_H \to \infty$, there exists $n(r)$ such that 
$$ |g_n|_G - M \cdot |\pi(g_n)|_H >  |x|_G + M \cdot |\pi(x)|_H $$ 
for all $n \geq n(r)$.
We thus obtain for $n \geq n(r)$ that
\begin{align*}
d_G(x,g_n) & \geq |g_n|_G - |x|_G > M \cdot \big( |\pi(g_n)|_H + |\pi(x)|_H \big) \\
& \geq M \cdot d_H (\pi(g_n), \pi(x) ) .
\end{align*}
So $D(g_n,x) = d_G(g_n,x)$, implying that $b_{g_n}^D(x) \leq b_{g_n}^G(x)$ for all $n \geq n(r)$.
Taking $n \to \infty$ implies that $F(x) \leq f(x) \leq -r$.
This holds for arbitrary $r$, so $F$ is unbounded from below in Case I.

%
%

{\bf Case II.}
$$ \liminf_{n \to \infty} \big( |g_n|_G - M \cdot |\pi(g_n)|_H \big)  = - \infty $$

As in Case I, without loss of generality, we can assume that $M \cdot |\pi(g_n)|_H - |g_n|_G \to \infty$.
Since $d_H$ is assumed to be geodesic, by Lemma \ref{lem:integer valued geodesic metrics} we can assume without loss of generality 
that $b_{\pi(g_n)}^H \to h \in \overline{ (H,d_H)}$ and $h$ is unbounded from below.

Fix $r \in \N$.  Choose $x \in G$ such that $h(\pi(x)) \leq -r$.
Since $M \cdot |\pi(g_n)|_H - |g_n|_G \to \infty$, there exists $n(r)$ such that 
$M \cdot |\pi(g_n)|_H-|g_n|_G>|x|_G+M \cdot |\pi(x)|_H$ for all $n \geq n(r)$.
As in Case I this implies that $D(g_n,x) = M \cdot d_H(\pi(g_n), \pi(x))$ for all $n \geq n(r)$,
which in turn implies that $F(x) \leq M \cdot h(\pi(x)) \leq - M r$.
This holds for arbitrary $r$, so $F$ is unbounded from below in Case II.

{\bf Case III.}
$$  \exists \ R > 0 \qquad  \Big|  |g_n|_G - M \cdot |\pi(g_n)|_H \Big|  \leq R $$
for all large enough $n$.

Since $|g_n|_D \to \infty$ it must be that $|g_n|_G \to \infty$ and $|\pi(g_n)|_H \to \infty$.
As in the first two cases we can pass to a subsequence using Lemma \ref{lem:integer valued geodesic metrics},
to assume without loss of generality that 
$$ b_{g_n}^G \to f \in \overline{(G,d_G)} \AND b_{\pi(g_n)}^H \to h \in \overline{ (H,d_H)} $$
and $f,h$ are unbounded from below.

Now fix some $r \geq 4R$.
Choose $x \in G$ such that $|x|_G = r = -f(x)$, using Lemma \ref{lem:integer valued geodesic metrics}.
By passing to a subsequence we may assume without loss of generality that $b_{g_n}^G(x) = -r$ for all $n$.

We now have two further cases:

{\bf Case III(a).}
$M \cdot b_{\pi(g_n)}^H(\pi(x)) \leq \frac{r}{2}$ for all large enough $n$.

Set $\rho = \lfloor \frac{3r}{4C} \rfloor$ and $z_n = x^{-1} g_n$.
Since $|\pi(z_n)|_H \to \infty$, 
using Lemma \ref{lem:integer valued geodesic metrics},
by passing to a further subsequence we can assume without loss of generality 
that there exists $q \in H$ such that $|q|_H = \rho = - b_{\pi(z_n)}^H(q)$ for all $n$.

Now, take $y \in G$ such that $\pi(y) = q$ and $|y|_G \leq C |q|_H$.
Thus,
$$ d_H(\pi(xy), \pi(g_n)) = d_H(q, \pi(z_n)) = d_H(\pi(x), \pi(g_n)) - \rho . $$
On the one hand we have that
$$ d_G(xy,g_n) \leq d_G(x,g_n) + |y|_G \leq |g_n|_G - r + C \rho \leq |g_n|_D - \tfrac{r}{4}  , $$
while the other hand we have that for all large enough $n$,
\begin{align*}
M \cdot d_H(\pi(xy), \pi(g_n)) & \leq M \cdot d_H( \pi(g_n), \pi(x) ) - M \rho 
\leq M \cdot |\pi(g_n)|_H +  \tfrac{r}{2} - M \rho \\
& \leq |g_n|_D - \tfrac{r}{4} + M .
\end{align*}
Taking a maximum over these two inequalities, and a limit as $n \to \infty$ we arrive at $F(xy) \leq M - \tfrac{r}{4}$.
As this holds for arbitrary $r$, we obtain that $F$ is unbounded from below in Case III(a).

{\bf Case III(b).} $M \cdot b_{\pi(g_n)}^H(\pi(x)) > \frac{r}{2}$ for infinitely many $n$.

In this case we have that for infinitely many $n$,
$$ M  \cdot d_H(\pi(g_n),\pi(x))>M \cdot |\pi(g_n)|_H+\frac{r}{2} \geq |g_n|_G-R+\frac{r}{2} = d_G(g_n,x)-R+\frac{3r}{2}>d_G(g_n,x) $$
which implies that
$D(g_n,x)=M \cdot d_H(\pi(g_n),\pi(x))$, so
$$D(g_n,x)>M \cdot |\pi(g_n)|_H+\frac{r}{2} \geq |g_n|_D-R+\frac{r}{2} \geq |g_n|_D+\frac{r}{4} .$$
 By passing to the limit we get that $F(x) \geq \frac{r}{4}$.
This holds for arbitrary $r$, so $F$ is unbounded from above in Case III(b).

Cases I, II, III(a), III(b) together complete the proof of the fifth property of Definition \ref{dfn:Banach metric}. 
\end{proof}

\begin{example} \label{exm:Cayley graphs and quotients}
One example of metrics on $G$ and $H \cong G/N$ satisfying the assumption of Lemma \ref{lem:BM construction} is as follows.

Fix some finite symmetric generating set $S$ for $G$ and let $d_G = d_S$ be the metric arising from the Cayley graph
with respect to $S$.  Note that since $\pi : G \to H$ is a surjective homomorphism, the set $\pi(S) \subset H$ is a finite symmetric
generating set for $H$.  Let $d_H$ be the metric corresponding to the Cayley graph with respect to $\pi(S)$.

It is easy to verify that $|\pi(x)|_H \leq |x|_G$. Also, for any $q \in H$, write $q = \pi(s_1) \cdots \pi(s_n)$ for $n = |q|_H$ and 
$s_j \in S$.  Then $x = s_1 \cdots s_n$ satisfies that $|x|_G \leq n = |q|_H$ and $\pi(x) = q$.
\end{example}

\begin{lemma} \label{lem:BM from quotient}
%
%
%
Let $G$, $\pi:G \to H$, $d_G,d_H,C,M$ and $D$ be as in Lemma \ref{lem:BM construction}
(with the same assumptions).  Recall that we assume $M>C$.

Then for any $h \in \p (H,d_H)$ there exists $f \in \p (G,D)$ such that $f(x) = M \cdot h(\pi(x))$ for all $x \in G$.

Specifically, $|G.f| \leq |H.h|$.
\end{lemma}

\begin{proof}
We use the notation $|x|_D$, $|x|_G$, $|q|_H$, $b_x^D$,  $b_x^G$, $b_q^H$ as 
in the proof of Lemma \ref{lem:BM construction}.

Let $h \in \p (H,d_H)$, and choose a sequence $(q_n)_n$ in $H$ such that $b_{q_n}^H \to h$.
Note that by Lemma \ref{lem:boundary points are at infinity}
it must be that $|q_n|_H \to \infty$.
Since $d_H$ is assumed to be a geodesic metric, by Lemma \ref{lem:integer valued geodesic metrics}, $h$ is unbounded from below.

For each $n$ choose $x_n \in G$ such that $\pi(x_n) = q_n$ and $|x_n|_G \leq C|q_n|_H$.  
Since $M \cdot |q_n|_H = M \cdot |\pi(x_n)|_H \leq |x_n|_D$, by passing to a subsequence we may assume without
loss of generality that $b_{x_n}^D \to f \in \overline{(G,D)}$.

We assumed that $M > C$, so $M \cdot |\pi(x_n)|_H > C |\pi(x_n)|_H \geq |x_n|_G$, so that $|x_n|_D = M \cdot |\pi(x_n)|_H$
for all $n$.  Also, for any $x \in G$ we have
\begin{align*}
M \cdot d_H(\pi(x) , \pi(x_n)) & \geq M \cdot |\pi(x_n)|_H - M \cdot |\pi(x)|_H \\
& \geq |x_n|_G + (M-C) |q_n|_H - M \cdot | \pi(x)|_H \\
& \geq d_G(x,x_n) + (M-C) |q_n|_H - M \cdot |\pi(x)|_H - |x|_G 
\end{align*}
for all $n$.  This implies that for all large enough $n$ (as soon as 
$(M-C) |q_n|_H > M \cdot |\pi(x)|_H + |x|_G$),
we have $b_{x_n}^D(x) = M \cdot b_{q_n}^H (\pi(x))$.
Hence $f(x) = M \cdot h(\pi(x))$ for all $x \in G$.

Note that for any $x \in G$ the function $b^D_x$ is bounded below
(by $-|x|_D$).  Since $h$ is unbounded from below, we also have that $f=M \cdot h \circ \pi$ is unbounded from below.  This implies that $f \neq b_x^D$ for any $x \in G$, so that $f \in \p (G,D)$.

This proves the first assertion.

For the second assertion, consider the map $x.f \mapsto \pi(x).h$.
The identity 
$$ x.f(y) = M \cdot (h(\pi(x^{-1} y)) - h(\pi(x^{-1}))) = M \cdot (\pi(x).h) (\pi(y)) . $$
shows that the map is well defined and injective.
Hence $|G.f| \leq |H.h|$.
\end{proof}

\subsection{Virtual homomorphisms}

\label{scn:detection}

Recall the notion of a {\em virtual homomorphism}, Definition \ref{dfn:vir hom}.

Also, throughout the text we use the notation $g^\gamma = \gamma^{-1} g \gamma$ for group elements $g,\gamma$,
as well as $A^g = \{ a^g \ : \ a \in A \}$ for a subset $A$ of a group and an element $g$.

The following lemma is well known, the proof is included for completeness.

\begin{lemma} \label{lem:VH implies VA}
Let $G$ be a finitely generated infinite group and assume that $G$ 
admits a virtual homomorphism. 

Then $G$ has an infinite virtually abelian quotient. 
\end{lemma}

\begin{proof}
A virtual homomorphism implies the existence of a finite index normal subgroup $H \lhd G$ such that $[G:H] < \infty$ and
some $K \lhd H$ with $H/K \cong \Z$. 
Consider the normal subgroup $N 
=\cap_{g \in G} K^g$. 
First, $[G/N:H/N]=[G:H]<\infty$, so we only need to show that $H/N$ is abelian. 
Note that for any $g \in G$ we have that $K^g \lhd H$. Also, $H/K^g \cong \Z$ for every $g \in G$. 
Therefore $H'=[H,H] \leq K^g$ for every $g \in G$, and thus also $H' \leq N=\cap_{g \in G} K^g$, which implies that $H/N$ is abelian. 
\end{proof}

We now discuss Cayley graphs, with a focus on virtually abelian groups.

If $S$ is a finite symmetric generating set for a group $G$, we denote by $\Gamma(G,S)$ the Cayley graph
with respect to $S$, by $d_S$ the corresponding metric, and by $|x|_S = d_S(x,1)$.
We also use the notation $\p \Gamma(G,S) = \p (G,d_S)$.

A Cayley graph $\Gamma(G,S)$ provides a metric $d_S$ which is a geodesic metric.
$\Gamma(G,S)$ always contains  geodesic rays (a sequence $\gamma= (\gamma_n)_n$ is a 
geodesic ray if every finite subsequence $(\gamma_k, \gamma_{k+1}, \ldots, \gamma_{k+n})$ is a geodesic).
It is not too difficult to prove (see \eg \cite[Lemma 3.2]{RY23} for a proof) that 
any geodesic ray $\gamma$ converges to a boundary point $b_{\gamma_n} \to \gamma_\infty \in \p (G,d_S)$.
Such points which happen to be limits of geodesics are called \define{Busemann points}, and are the subject of a different discussion,
see \eg references in \cite{RY23, Walsh}.
It is also easy to see that if $\gamma$ is a geodesic ray, then $x.\gamma$ given by $(x.\gamma)_n = x \gamma_n$
is also a geodesic, which converges to $x.\gamma_\infty$.

Walsh provides a characterization of which geodesics converge to the same boundary point:

\begin{proposition}[Proposition 2.1 in \cite{Walsh}]
\label{prop:Walsh geodesics}
Let $\Gamma$ be a Cayley graph.
Two geodesic rays $\alpha, \beta$ converge to the same boundary point $\alpha_\infty = \beta_\infty \in \p \Gamma$,
if and only if there exists a geodesic ray $\gamma$ such that $|\gamma \cap \alpha| = |\gamma \cap \beta| = \infty$.
\end{proposition}
 
Here we slightly abuse notation and denote by 
$\alpha \cap \gamma = \{ \alpha_n \ : \ \exists \ k \in \N \ , \ \alpha_n = \gamma_k \}$
the set of elements which are in both geodesics $\alpha$ and $\gamma$.

\begin{example} \label{exm:Zd boundary}
Consider $G = \Z^d$ with the standard generating set $S = \IP{ \pm e_1, \ldots, \pm e_d }$ (the standard basis and their inverses).
The space $\overline{(G ,d_S)}$ in this case is composed of all functions of the form 
$$ h_{\alpha_1, \ldots , \alpha_d}(z_1 ,\ldots, z_d) = \sum_{j=1}^d h_{\alpha_j}(z_j) $$
where $\alpha_1 , \ldots, \alpha_d \in \Z \cup \{ - \infty, \infty \}$
$$ h_\alpha : \Z \to \Z \qquad h_\alpha(z) = 
\begin{cases}
|\alpha-z| - |\alpha| & \textrm{ if } \alpha \in \Z , \\
-z & \textrm{ if } \alpha = \infty , \\
z & \textrm{ if } \alpha = - \infty 
\end{cases}
$$
For $\vec \alpha \in \Z^d$ we have that $h_{\vec \alpha} = b_{\vec \alpha}$.
So $h_{\alpha_1,\ldots, \alpha_d} \in \p (G,d_S)$ if and only if there exists some $1 \leq j \leq d$ such that 
$\alpha_j \in \{- \infty, \infty\}$.

The action of the group on $\overline{ (G,d_S)}$ is given by $\vec z . h_{\alpha_1,\ldots, \alpha_d}
= h_{\alpha_1 + z_1, \ldots, \alpha_d + z_d}$, with the convention $\infty + z = \infty$ and $-\infty + z = -\infty$.

Note that if $\alpha_j \in \{- \infty , \infty\}$ for all $1 \leq j \leq d$, then $h_{\alpha_1,\ldots, \alpha_d}$ is a fixed point in
the boundary.

See Figure \ref{fig:Z2 boundary}.

Specifically in this case, all boundary points are limits of geodesics.
For example, the point $h_{\infty, \ldots, \infty}$ is obtained from the limit of any geodesic ray whose coordinates 
all tend to $\infty$. \eg $\gamma_{dn+j} = n \vec 1 + e_1 + \cdots + e_j \in \Z^d$ for $n \in \N$ and $0 \leq j < d$, defines 
such a geodesic.

See \cite{develin} for much more on metric-functional boundaries of abelian groups.
\end{example}

\begin{figure}[ht] \label{fig:Z2 boundary}
\includegraphics[width=0.7\textwidth]{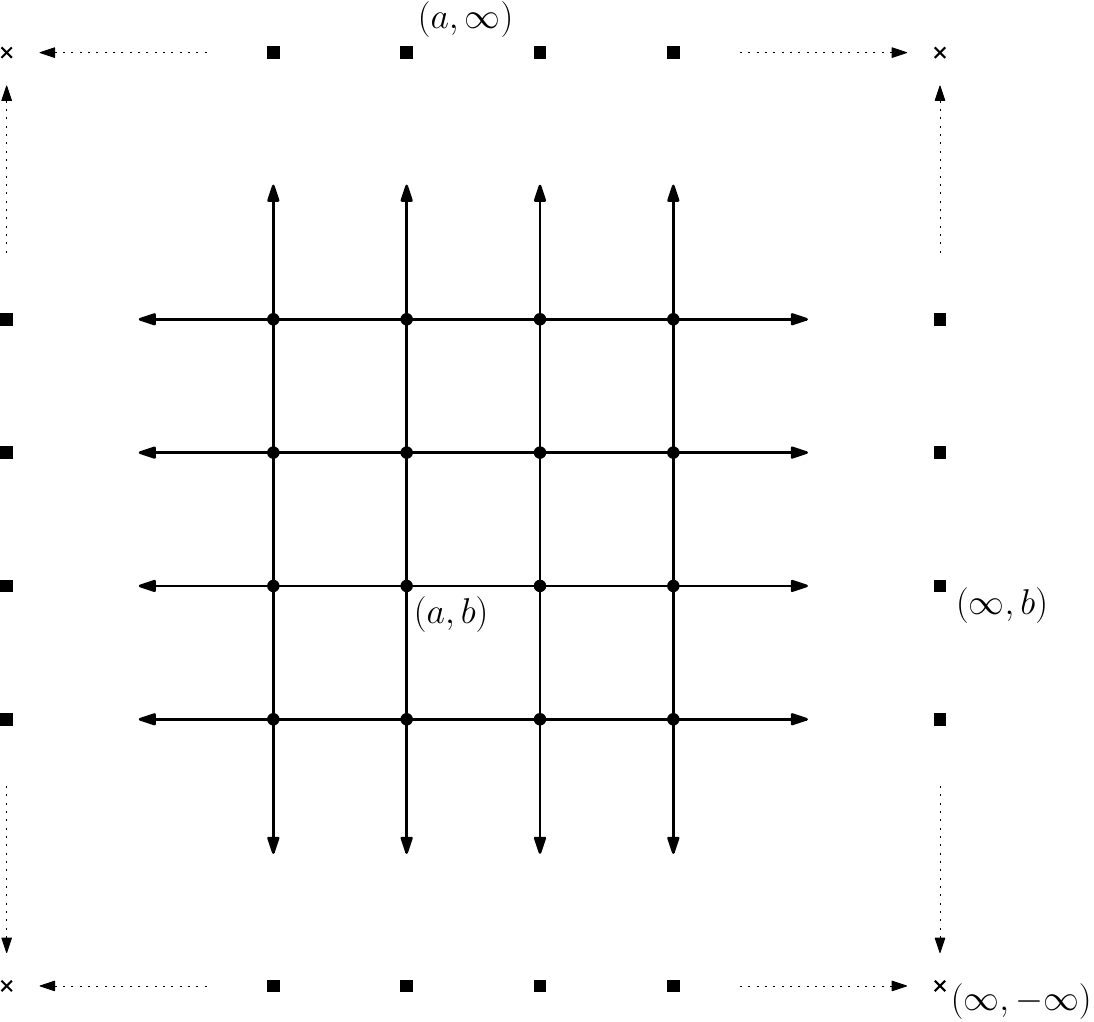}
\caption{The standard Cayley graph of $\Z^2$ with its metric-functional boundary.}
\end{figure}

\begin{lemma} \label{lem:Cayley graphs for VA}
Let $G$ be a finitely generated infinite virtually abelian group, and let $S$ be a finite symmetric generating set for $G$.

There exists a finite index subgroup $N \leq G$, $[G:N] < \infty$, 
such that $N \cong \Z^d$ for some $d \geq 1$, and there exists a finite symmetric generating set $U$ for $N$ such that the following hold:
\begin{itemize}
\item $S \cap U=\emptyset$
\item Denoting $T=S \uplus U$, we have that $|x|_U = |x|_T$ for any $x \in N$.
\end{itemize}
In particular, every geodesic in $\Gamma(N,U)$ is a geodesic in $\Gamma(G,T)$.
\end{lemma}

\begin{proof}
Since $G$ is finitely generated, infinite and virtually abelian, 
there exists a finite index normal subgroup $H \cong \mathbb{Z}^d$ in $G$, 
for $d \geq 1$ (see \eg \cite[Theorem 1.5.2]{HFbook} for a proof). 
Let $R$ be a finite set of representatives for the cosets of $H$ such that $1 \in R$; that is $G = \uplus_{r \in R} Hr$.
For every $g \in G$ let $x_g \in H$ and $r_g \in R$ be the unique elements such that $g=x_g r_g$.

Let $b_1, \ldots, b_d \in H$ be elements that are mapped to the standard basis of $\Z^d$ under the implicit 
isomorphism $H \cong \Z^d$.
Note that $B=\{b_1^{\pm 1},\ldots,b_d^{\pm 1}\}$ is a finite symmetric generating set for $H$.

Define:
$$F=\{(x_{sr})^g  \ | \  s \in R \cup S \ , \  r \in R \ , \ g \in G\} . $$
Since $H$ is abelian, for every $x \in H$ it holds that $H \leq C_G(x)=\{g \in G \ | \  x^g=x\}$, so 
$[G:C_G(x)] < \infty$ for every $x \in H$. 
Thus, every $x \in H$ has a finite orbit under conjugation, by the orbit-stabilizer theorem.
Hence, $F$ is a finite set (because $R,S$ are finite).
So we can define
$$ M=\max \{|x|_B \ | \  x \in F\}  . $$

{\bf Step I.}
First, we prove by induction on $n$ that for any $r_1, \ldots, r_n \in R$ such that $r_1 \cdots r_n \in H$ it holds that 
\begin{align}
\label{eqn:product of rs}
|  r_1 \cdots r_n |_B & \leq M n  .
\end{align}

For $n=1$ this is obvious as $r_1 \in H \cap R$ implies that $r_1=1$ so $|r_1|_B = 0$.

Now, for $n \geq 1$,
if $r_1 \cdots r_{n+1} \in H$ then we can write $y = x_{r_1 r_2}$ and $\rho = r_{r_1 r_2}$ 
(recall the decomposition $g=x_g r_g$ from above),
so that $\rho r_3 \cdots r_{n+1} = y^{-1} r_1 \cdots r_{n+1} \in H$.  By induction, we conclude that
$$ | r_1 \cdots r_{n+1}  |_B \leq | y  |_B + |  \rho r_3 \cdots r_{n+1}  |_B \leq M + Mn = M(n+1) $$
proving \eqref{eqn:product of rs}.

{\bf Step II.}
Next, we prove that for every $x \in H$ we have
\begin{align}
\label{eqn:infty < 2M S}
| x  |_B & \leq 2M \cdot |x|_S .
\end{align}
Indeed, let $x \in H$ and write $x=s_1 \cdots s_n$ for $n=|x|_S$ and $s_1,\ldots, s_n \in S$.
Denote $y_j=x_{s_j}$ and $r_j=r_{s_j}$ for every $1 \leq j \leq n$ (recall the decomposition $g=x_g r_g$ from above). 
Denote $q_1=1$,  $q_{j+1}= (r_1 \cdots r_j)^{-1}$ for $1 \leq j \leq n$, $z=x_{(q_{n+1})^{-1}}$ and $r=r_{(q_{n+1})^{-1} }$. 
With this notation we have:
$$x=y_1r_1 \cdots y_nr_n=(y_1)^{q_1} \cdots (y_n)^{q_n} \cdot (q_{n+1})^{-1} =(y_1)^{q_1} \cdots (y_n)^{q_n} \cdot z \cdot r . $$
Since $x \in H$, it follows that $r=1$, so
\begin{align}
\label{eqn:bound on xS}
| x |_B & \leq \sum_{j=1}^n |(y_j)^{q_j } |_B + |  z  |_B  \leq Mn+ | z  |_B .
\end{align}
Since $r=1$, we have that $H \ni z = zr = (q_{n+1})^{-1} = r_1 \cdots r_n$.
Using \eqref{eqn:product of rs} 
$$| z |_B = |  r_1 \cdots r_n  |_B  \leq M n = M |x|_S . $$
Together with \eqref{eqn:bound on xS}, this completes a proof of \eqref{eqn:infty < 2M S}.

{\bf Step III.} 
Now fix some integer $K>2M$ and let $N$ be the subgroup generated by $\tilde U = \{ b^K \ : \ b \in B \}$.
Since by \eqref{eqn:bound on xS} 
we have that $K |x|_{\tilde U} = |x|_B \leq 2M |x|_S$ for all $x \in N$, we obtain that $S \cap N \subset \{1\}$.

Also, note that $N$ is a normal subgroup, $N \lhd G$.
Indeed, for any $u \in  \tilde U$ there is some $b \in B$ such that $u = b^K$.
Also, for any $g \in G$ we know that $g^{-1} b g = b_1^{z_1} \cdots b_d^{z_d}$ for some integers 
$z_1, \ldots, z_d \in \Z$ (because $H$ is normal in $G$).
Therefore, using that $H$ is abelian,
$$ g^{-1} u g =   (b_1^{z_1} \cdots b_d^{z_d})^K = (b_1^K)^{z_1} \cdots (b_d^K)^{z_d} \in N . $$
This proves that $N \lhd G$.

Now define the symmetric generating set $U := \{ g^{-1} b^K g \ : \ b \in B \ , \ g \in G \}$.
As mentioned above, each orbit $\{ b^g \ : \ g \in G\}$  
is finite, so the generating set $U$ is finite, and thus induces a Cayley graph $\Gamma(N,U)$.

Define $T = S \uplus U$.
We now prove by induction on $n$ that for any $t_1, \ldots, t_n \in T$ such that $x=t_1 \cdots t_n \in N$
we have $|x|_U \leq n$.

The base case, $n=1$, is the case where $x=t_1 \in T \cap N$.
Since $S \cap N \subset \{1\}$, it must be that $x \in U \cup \{1\}$,
which implies that $|x|_U \leq 1$ as required.

For the induction step, assume that $n > 1$ and $t_1, \ldots, t_n \in T$ are such that $x = t_1 \cdots t_n \in N$.
We have a few cases.

{\bf Case I.} $t_1 , \ldots , t_n \in S$. In this case, using \eqref{eqn:bound on xS} we have that 
$$ K |x|_U \leq K |x|_{\tilde U} = |x|_B \leq 2M |x|_S \leq 2Mn , $$
implying that $|x|_U < n$.

{\bf Case II.} $t_n \in U$. In this case, $y=t_1 \cdots t_{n-1} \in N$, and by induction,
$|x|_U \leq |y|_U + |t_n|_U \leq n$.

{\bf Case III.} $t_1 \in U$.  This case is similar to the previous one.  We have that $y = t_2 \cdots t_n \in N$
so that $|x|_U \leq |t_1|_U + |y|_U \leq n$ by induction.

{\bf Case IV.} $t_1, t_n \in S$ and there exists some $2 \leq j \leq n-1$ for which $t_j \in U$.
In this case we can choose $2 \leq \ell \leq n-1$ such that $t_\ell \in U$ and $t_j \in S$ for all $\ell < j \leq n$.
Set $g= t_{\ell+1} \cdots t_n$ and $u = t_\ell$ and note that 
$$ x = t_1 \cdots t_n = t_1 \cdots t_{\ell-1} t_{\ell+1} \cdots t_n \cdot g^{-1} u g . $$
So $y = t_1 \cdots t_{\ell-1} t_{\ell+1} \cdots t_n \in N$ and by induction
$|x|_U \leq |y|_U + |g^{-1} u g |_U \leq n$, where we have used that $g^{-1} u g \in U$ as well.

This completes the induction step.

To complete the proof of the lemma, note that we have shown, in particular, that for any $x \in N$ 
we have $|x|_U \leq |x|_T$.
Since $U \subset T$, this implies that $|x|_U = |x|_T$ for any $x \in N$.
Specifically, any geodesic in $\Gamma(N,U)$ is also a geodesic in $\Gamma(G,T)$.
\end{proof}

\begin{corollary} \label{cor:vir Abelian fin orbit}
Let $G$ be a finitely generated infinite virtually abelian group.

There exists a Cayley graph $\Gamma(G,T)$ of $G$ with a finite orbit in $\partial \Gamma(G,T)$.
\end{corollary}

\begin{proof}
By Lemma \ref{lem:Cayley graphs for VA}, we can choose $N \leq G$ of finite index $[G:N] < \infty$ such that $N \cong \Z^d$, 
and we can also find $U \subseteq T$ two finite symmetric sets such that $G=\langle T \rangle$, $N=\langle U \rangle$ 
and such that every geodesic in $\Gamma(N,U)$ is also a geodesic in $\Gamma(G,T)$.
It is a result of Develin \cite{develin} that 
$\partial \Gamma(N,U)$ contains a Busemann point fixed under the action of $N$. (See Theorem 3 in \cite{develin}.  The invariant Busemann point arises from a geodesic corresponding to a vertex, 
or a $0$-dimensional face, of the appropriate polytope.)

We use $\gamma$ to denote a geodesic in $\Gamma(N,U)$ which converges to $h \in \p \Gamma(N,U)$ 
such that $x.h=h$ for all $x \in N$.

By Lemma \ref{lem:Cayley graphs for VA}, $\gamma$ is also a geodesic in $\Gamma(G,T)$,
and thus in the space $\overline{ (G,d_T) }$, $\gamma$ converges to some point $f \in \p \Gamma(G,T)$.

Fix any $x \in N$. 
We know that $x.h = h$, implying that the geodesic $x.\gamma$ converges to $h$ as well.
By Proposition \ref{prop:Walsh geodesics}, there exists some infinite geodesic $\alpha$ in $\Gamma(N,U)$ 
such that $|\alpha \cap \gamma| = |\alpha \cap x.\gamma| = \infty$.
By Lemma \ref{lem:Cayley graphs for VA}, $\alpha$ is also a geodesic in $\Gamma(G,T)$.
Using Proposition \ref{prop:Walsh geodesics} in the graph $\Gamma(G,T)$ we obtain that in
the space $\overline{ (G,d_T) }$, the geodesics $\gamma, x.\gamma$ converge to the same boundary point $f = x.f \in \p \Gamma(G,T)$.

Since this holds for any $x \in N$ we get that $N$ is contained in the stabilizer of $f$,
\ie $N \lhd \{ x \in G \ : \ x.f=f \}$.  Since $[G:N] < \infty$, we get that this stabilizer has finite index,
and thus, $|G.f| < \infty$.
This is the required finite orbit.
\end{proof}

We now move to prove Theorem \ref{thm:detection}, stating that $G$ admits a virtual homomorphism
if and only if there exists a Banach metric on $G$ with a finite orbit in the boundary.

\begin{proof}[Proof of Theorem \ref{thm:detection}]
As mentioned after Definition \ref{dfn:vir hom}, if $h \in \p (G,d)$ has a finite orbit, then $h$ is a virtual homomorphism.

For the other direction, we assume that $G$ admits some virtual homomorphism (\ie $G$ is virtually indicable).
By Lemma \ref{lem:VH implies VA}, there exists a surjective homomorphism $\pi : G \to H$ such that $H$ is virtually abelian.

Fix some Cayley graph metric $d_G$ on $G$, with $S$ the corresponding finite symmetric generating set.
Note that $\pi(S)$ is a finite symmetric generating set for $H$.
By Corollary \ref{cor:vir Abelian fin orbit}, there exists a finite symmetric generating set $T$ for the group $H$
such that $\p \Gamma(H,T)$ contains a finite orbit.  That is, for some $h \in \p \Gamma (H,T)$ we have $|H.h| < \infty$.
Let $d_H = d_T$.
Since $(H,d_H)$ and $(H,d_{\pi(S)})$ are two Cayley graphs of the same group,
there is some $C>0$ such that $d_{\pi(S)}(p,q) \leq C d_H(p,q)$ for all $p,q \in H$.
As explained in Example \ref{exm:Cayley graphs and quotients}, this shows that $G, H, \pi, d_G, d_H$ satisfy the assumptions
of Lemmas \ref{lem:BM construction} and \ref{lem:BM from quotient}, with this constant $C>0$.
Hence, by taking some integer $M> C$, using Lemmas \ref{lem:BM construction} and \ref{lem:BM from quotient},
we have a Banach metric $D$ on $G$ and some $f \in \p (G,D)$ such that $f(x) = M \cdot h(\pi(x))$ for all $x \in G$.
Specifically, $|G.f| \leq |H.h| < \infty$, providing us with the required finite orbit.
\end{proof}

\section{Finite index subgroups}

\label{scn:finite index}

In this section we prove Theorem  \ref{thm:finite index BM}, stating that a Banach metric induces 
a Banach metric on a finite index subgroup.

\begin{proof}[Proof of Theorem \ref{thm:finite index BM}]
Let $d_G$ be a Banach metric on a group $G$.
Let $H \leq G$ be a subgroup of finite index $[G:H] < \infty$.
Let $d(x,y) = d_G(x,y)$ for all $x,y \in H$, which is the metric on $H$ as a subspace of $G$.

The first three properties in Definition \ref{dfn:Banach metric} are immediate to verify.

The fourth property follows from the fact that quasi-isometry is an equivalence relation, 
and the fact that any Cayley graph on $H$ is quasi-isometric to a Cayley graph of $G$.
The identity map on $H$ into $G$ provides a quasi-isometry from $(H,d)$ to $(G,d_G)$.
This implies that $(H,d)$ is quasi-isometric to a Cayley graph of $G$, and thus to a Cayley graph of $H$.

To prove the fifth property in Definition \ref{dfn:Banach metric}, 
choose any $h \ni \p (H,d)$.
Then $b_{x_n} \big|_H \to h$ for some sequence $x_n \in H$.
By perhaps passing to a subsequence, we may assume without loss of generality that $b_{x_n} \to f \in \overline{ (G,d_G) }$,
and we find that $f \big|_H = h$.

If $f \in \p (G,d_G)$ then it is unbounded by the fifth property in Definition \ref{dfn:Banach metric}.
Since $d_G$ is proper, any $b_x$ is also unbounded.
Hence $f$ is unbounded in any case.
%
So, we can find a sequence $(g_n)_n$ in $G$ such that 
$|f(g_n) | \to \infty$.  Since $[G:H] < \infty$ the sequence $(g_n)_n$ must be in some coset of $H$ infinitely many times,
implying that by passing to a subsequence we may assume that $g_n = y_n r$ for $y_n \in H$ and fixed $r \in G$.
By the Lipschitz property we have that $|f(y_n) - f(g_n)| \leq |r|$ for all $n$, so that $|h(y_n)| = |f(y_n)| \geq |f(g_n)| - |r| \to \infty$,
implying that $h$ is unbounded.
%
\end{proof}

\section{No detection in hyperbolic groups}

\label{scn:no detection}

In this section we prove Theorem \ref{thm:hyperbolic no detection}, stating that in 
any Cayley graph of a non-virtually cyclic Gromov hyperbolic group, there does not exists a finite orbit in the boundary.
A good source for the definitions and properties stated in this section is \cite[Chapter 11]{DK18}.
(We will work in a less general framework, as we only require some basic properties for our purposes.)

Let $G$ be a finitely generated group, and fix $d=d_S$ a Cayley metric on $G$.

For $x,y \in G$ define the {\em Gromov product} $(x,y) : = \tfrac12 ( |x| + |y| - d(x,y)  )$.

Let $\Omega_\infty \subset G^\N$ be the set of all sequences $(x_n)_n$ such that
$$ \lim_{n,m \to \infty} (x_n,x_m) = \infty $$
Define a relation $\sim$ on $\Omega_\infty$ by declaring $(x_n)_n \sim (y_n)_n$ if 
$$ \lim_{n \to \infty} (x_n,y_n) = \infty $$
Two sequences $\mbf{x}, \mbf{y} \in \Omega_\infty$ are \define{equivalent} if there is a finite 
sequence $\mbf{x}^{(0)}, \mbf{x}^{(1)}, \ldots, \mbf{x}^{(n)} \in \Omega_\infty$ such that $\mbf{x}^{(j)} \sim \mbf{x}^{(j+1)}$ 
for all $j$ and $\mbf{x}^{(0)} = \mbf{x}$ and $\mbf{x}^{(n)} = \mbf{y}$.
Since the relation $\sim$ is symmetric and reflexive, the above defines an equivalence relation on $\Omega_\infty$.
For a sequence $\mbf{x} = (x_n)_n \in \Omega_\infty$, we write $[\mbf{x}]$ for the equivalence class of $\mbf{x}$.
We also write $[\mbf{x}] = \lim x_n$.

The set of equivalence classes is denoted by $\p_{\Gr} G = \{ [\mbf{x}] \ : \ \mbf{x} \in \Omega_\infty \}$,
and is called the {\em Gromov boundary}. (This boundary is invariant under quasi-isometries, see \eg \cite{DK18}, 
therefore we do not denote a dependence on the Cayley metric.  However, we do not require this property for our purposes.)

It is a simple exercise to prove that if $\gamma$ is an infinite geodesic then $\gamma \in \Omega_\infty$.
The next lemma implies that, in fact,  any sequence in $\Omega_\infty$ is equivalent to some geodesic.

\begin{lemma} \label{lem:geodesics in Gromov boundary}
For any $\mbf{x} \in \Omega_\infty$ there exists an infinite geodesic $\gamma$ such that $[\gamma] = [\mbf{x}]$.
\end{lemma}

\begin{proof}
Let $\mbf{x} = (x_n)_n$ be such that $|x_n| \to \infty$.
For each $n$ let $\gamma^{(n)}$ be a finite geodesic from $1$ to $x_n$.
Let $J_0 = \N$ and $\alpha_0=1$.  Note that for all $n \in J_0$ we have $\gamma^{(n)}_0 = \alpha_0$.
There exists $|\alpha_1| = 1$ and an infinite subset $J_1 \subset J_0$, $|J_1|=\infty$, 
such that for all $n \in J_1$ we have $\gamma^{(n)}_1 = \alpha_1$.
Continuing inductively, we find a decreasing sequence of infinite subsets $J_{k+1} \subset J_k \subset \cdots \subset J_0 = \N$,
and an infinite geodesic $\alpha = (\alpha_n)_n$, such that for all $k \in \N$ and all $n \in J_k$ it holds that $\gamma^{(n)}_j = \alpha_j$
for any $0 \leq j \leq k$.
(All this is basically just compactness of $\{0,1\}^G$ with pointwise convergence.)

For each $k \in \N$ write $J_k = \{ m_{k,0} < m_{k,1} < m_{k,2} < \cdots \}$.
Define $y_k : = x_{m_{k,k}}$, so $\mbf{y} = (y_n)_n$ is a subsequence of $(x_n)_n$.
It is a simple exercise to show that since $\mbf{y}$ is a subsequence, $[\mbf{y} ] = [\mbf{x}]$.

Now, consider $n = m_{k,k} \in J_k$. 
Since $\gamma^{(n)}$ is a geodesic from $1$ to $y_k$ through $\alpha_k$, we have that $d(y_k, \alpha_k) = |y_k| - |\alpha_k|$.
Thus, $(y_k,\alpha_k) = |\alpha_k| \to \infty$ as $k \to \infty$.
This implies that $\mbf{y} \sim \alpha$, so that $[\alpha] = [\mbf{x}]$.
\end{proof}

$G$ is called \define{Gromov hyperbolic} if there exists $\delta$ such that for all $x,y,z \in G$ we have 
$$ (x,y) \geq \min \{ (x,z) , (z,y) \} - \delta . $$

We now prove Theorem \ref{thm:hyperbolic no detection}, stating that the metric functional boundary of a Cayley graph
cannot detect virtual homorphisms in Gromov hyperbolic groups.

\begin{proof}[Proof of Theorem \ref{thm:hyperbolic no detection}]
Let $d$ be a Cayley metric on $G$.

Note that $G$ acts on $\p_{\Gr} G$ by $g.[(x_n)_n] = [(gx_n)_n]$.
One easily verifies that this is a well defined action of $G$ (because $(gx,gy) \geq (x,y) - |g|$).

Define a map from $\p (G,d)$ to $\p_{\Gr} G$ as follows:
Given $h \in \p (G,d)$ choose some sequence $(x_n)_n$ such that $b_{x_n} \to h$, and then map $h \mapsto [(x_n)_n]$.
For any $y \in G$,
$$ (x_n,x_m) \geq - \tfrac12 ( b_{x_n}(y) + b_{x_m}(y) ) \to - h(y)  $$
as $n,m \to \infty$.
Since $(G,d)$ is geodesic, by Lemma \ref{lem:integer valued geodesic metrics}, for any $r>0$ there exists $y_r \in G$ with 
$h(y_r) = - |y_r| = -r$.  This shows that $b_{x_n} \to h$ implies that $(x_n)_n \in \Omega_\infty$.

Moreover, similarly to the above, if $b_{x_n} \to h$ and $b_{z_n} \to h$, then for any $y \in G$ we have 
$$ (x_n,z_n) \geq - \tfrac12 (b_{x_n}(y) + b_{z_n}(y)) \to -h(y) $$
as $n \to \infty$.
Since $-h(y_r) \to \infty$ as $r \to \infty$, we conclude that $(x_n)_n \sim (z_n)_n$.  Thus, the map $h \mapsto [(x_n)_n]$ is well defined.

Equivariance of this map is easy to verify.

The map is surjective by Lemma \ref{lem:geodesics in Gromov boundary}, since any geodesic $\gamma$ 
converges in $\overline{(G,d)}$ to some metric-functional in $\p (G,d)$ (see \eg \cite{develin}).

This completes a proof that $\p (G,d)$ can be mapped equivariantly and surjectively onto $\p_{\Gr} G$.
(Such a surjective map is well known, see \eg \cite{WW05}.  However,
in many texts only the hyperbolic case is considered, and also the group action is not mentioned.  Since the proof is short, we included all the details here.)

It is well known that if $G$ is Gromov hyperbolic then any orbit in $\p_{\Gr} G$ is dense ($G$ acts {\em minimally} on $\p_{\Gr} G$),
see \eg \cite[Proposition 4.2]{KB02}. 
Also, the only case where the Gromov boundary $\p_{\Gr} G$ is finite is if $G$ is virtually cyclic
(see \eg \cite[Lemma 11.130]{DK18}).

So, if $G$ is a non-virtually cyclic Gromov hyperbolic group, the Gromov boundary has no finite orbits,
implying (by the surjective equivariant map from $\p (G,d)$ onto the Gromov boundary) that there are no finite orbits in 
$\p (G,d)$.
\end{proof}


\begin{thebibliography}{10}

\bibitem{arosio2024horofunction}
L.~Arosio, M.~Fiacchi, S.~Gontard, and L.~Guerini.
\newblock The horofunction boundary of a gromov hyperbolic space.
\newblock {\em Mathematische Annalen}, 388(2):1163--1204, 2024.

\bibitem{Bass72}
H.~Bass.
\newblock The degree of polynomial growth of finitely generated nilpotent
  groups.
\newblock {\em Proceedings of the London Mathematical Society}, 3(4):603--614,
  1972.

\bibitem{BT24}
C.~Bodart and K.~Tashiro.
\newblock Horofunctions on the {H}eisenberg and {C}artan groups.
\newblock {\em arXiv preprint arXiv:2407.11943}, 2024.

\bibitem{busemann2005geometry}
H.~Busemann.
\newblock {\em The geometry of geodesics}.
\newblock Courier Corporation, 2005.

\bibitem{advancedproblems}
L.~Carlitz, A.~Wilansky, J.~Milnor, R.~Struble, N.~Felsinger, J.~Simoes,
  E.~Power, R.~Shafer, and R.~Maas.
\newblock Advanced problems: 5600-5609.
\newblock {\em The American Mathematical Monthly}, 75(6):685--687, 1968.

\bibitem{Gromoll}
J.~Cheeger and D.~Gromoll.
\newblock The splitting theorem for manifolds of nonnegative ricci curvature.
\newblock {\em Journal of Differential Geometry}, 6(1):119--128, 1971.

\bibitem{develin}
M.~Develin.
\newblock {C}ayley compactifications of abelian groups.
\newblock {\em Annals of Combinatorics}, 6(3):295--312, 2002.

\bibitem{DK18}
C.~Dru{\c{t}}u and M.~Kapovich.
\newblock {\em Geometric group theory}, volume~63.
\newblock American Mathematical Soc., 2018.

\bibitem{Grigorchuk80}
R.~I. Grigorchuk.
\newblock Burnside problem on periodic groups.
\newblock {\em Funktsional'nyi Analiz i ego Prilozheniya}, 14(1):53--54, 1980.

\bibitem{Grigorchuk84}
R.~I. Grigorchuk.
\newblock Degrees of growth of finitely generated groups, and the theory of
  invariant means.
\newblock {\em Izvestiya Rossiiskoi Akademii Nauk. Seriya Matematicheskaya},
  48(5):939--985, 1984.

\bibitem{Grigorchuk90}
R.~I. Grigorchuk.
\newblock On growth in group theory.
\newblock In {\em Proceedings of the International Congress of Mathematicians},
  volume~1, pages 325--338, 1990.

\bibitem{Gromov81}
M.~Gromov.
\newblock Groups of polynomial growth and expanding maps.
\newblock {\em Publications Math{\'e}matiques de l'IH{\'E}S}, 53:53--78, 1981.

\bibitem{gromov1981hyperbolic}
M.~Gromov.
\newblock Hyperbolic manifolds, groups and actions.
\newblock In {\em Riemann surfaces and related topics: Proceedings of the 1978
  Stony Brook Conference (State Univ. New York, Stony Brook, NY, 1978)},
  volume~97, pages 183--213, 1981.

\bibitem{Guivarch73}
Y.~Guivarc'h.
\newblock Croissance polynomiale et p{\'e}riodes des fonctions harmoniques.
\newblock {\em Bulletin de la soci{\'e}t{\'e} math{\'e}matique de France},
  101:333--379, 1973.

\bibitem{KB02}
I.~Kapovich and N.~Benakli.
\newblock Boundaries of hyperbolic groups.
\newblock {\em arXiv preprint math/0202286}, 2002.

\bibitem{karlsson2001non}
A.~Karlsson.
\newblock Non-expanding maps and busemann functions.
\newblock {\em Ergodic Theory and Dynamical Systems}, 21(5):1447--1457, 2001.

\bibitem{karlsson2008ergodic}
A.~Karlsson.
\newblock Ergodic theorems for noncommuting random products.
\newblock {\em Lecture notes available on the author’s website}, 2008.

\bibitem{karlsson2021linear}
A.~Karlsson.
\newblock From linear to metric functional analysis.
\newblock {\em Proceedings of the National Academy of Sciences},
  118(28):e2107069118, 2021.

\bibitem{karlsson2021hahn}
A.~Karlsson.
\newblock Hahn-banach for metric functionals and horofunctions.
\newblock {\em Journal of Functional Analysis}, 281(2):109030, 2021.

\bibitem{karlsson2024metric}
A.~Karlsson.
\newblock A metric fixed point theorem and some of its applications.
\newblock {\em Geometric and Functional Analysis}, pages 1--26, 2024.

\bibitem{Milnor68}
J.~Milnor.
\newblock Growth of finitely generated solvable groups.
\newblock {\em Journal of Differential Geometry}, 2(4):447--449, 1968.

\bibitem{Gabor}
G.~Pete.
\newblock {\em Probability and geometry on groups}.
\newblock available on author's website.

\bibitem{RY23}
L.~Ron-George and A.~Yadin.
\newblock Groups with finitely many {B}usemann points.
\newblock {\em accepted, Groups, Geometry and Dynamics}, 2023.
\newblock arXiv:2305.02303.

\bibitem{Walsh}
C.~Walsh.
\newblock The action of a nilpotent group on its horofunction boundary has
  finite orbits.
\newblock {\em Groups, Geometry, and Dynamics}, 5(1):189--206, 2011.

\bibitem{WW05}
C.~Webster and A.~Winchester.
\newblock Boundaries of hyperbolic metric spaces.
\newblock {\em Pacific journal of mathematics}, 221(1):147--158, 2005.

\bibitem{Wolf68}
J.~A. Wolf.
\newblock Growth of finitely generated solvable groups and curvature of
  {R}iemannian manifolds.
\newblock {\em Journal of Differential Geometry}, 2(4):421--446, 1968.

\bibitem{HFbook}
A.~Yadin.
\newblock {\em Harmonic Functions and Random Walks on Groups}.
\newblock Cambridge University Press, 2024.

\end{thebibliography}
\end{document}